\documentclass[11pt,a4paper,reqno]{amsart}
\usepackage{amsthm,amssymb,amsmath,graphicx,enumerate}
\usepackage{mathrsfs,setspace,pstricks,multicol,latexsym}
\usepackage{listings}
\usepackage{color} 
\definecolor{mygreen}{RGB}{28,172,0} 
\definecolor{mylilas}{RGB}{170,55,241}

\usepackage{graphicx}
\usepackage{amsmath}
\usepackage[per-mode=symbol,detect-weight]{siunitx}
\usepackage{amssymb}
\usepackage{epstopdf}
\usepackage{a4wide}
\usepackage{appendix}
\usepackage{hyperref}
\usepackage{color}
\newcolumntype{P}[1]{>{\centering\arraybackslash}p{#1}}
\usepackage{subfigure}
\usepackage{graphicx,graphics}
\usepackage{hyperref}
\newcommand{\abs}[1]{ \left\lvert#1\right\rvert}
\newcommand{\norm}[1]{\left\lVert#1\right\rVert}
\usepackage[vlined,ruled,linesnumbered]{algorithm2e}
\usepackage{mathtools}
\usepackage{diagbox}

\newtheorem{theorem}{Theorem}

\newtheorem{lemma}[theorem]{Lemma}

\newtheorem{remark}[theorem]{Remark}

\begin{document}
\author{Anindya Goswami}
\author{Kuldip Singh Patel*}\thanks{*Corresponding Author}
\author{Pradeep Kumar Sahu}
\title[Stability of variable coefficient PDEs]{A novel difference equation approach for the stability and robustness of compact schemes for variable coefficient PDEs}
\thanks{Anindya Goswami, Department of Mathematics, IISER Pune, India, anindya@iiserpune.ac.in,\\ Kuldip Singh Patel, Department of Mathematics, IIT Patna, Bihta, 801106, Bihar, India, kspatel@iitp.ac.in,\\ Pradeep Kumar Sahu, Department of Mathematics, IIT Patna, Bihta, 801106, Bihar, India, pradeep$\_$2321ma03@iitp.ac.in.}
\begin{abstract}
Fourth-order accurate compact schemes for variable coefficient convection diffusion equations are considered. A sufficient condition for the stability of the fully discrete problem is derived using a difference equation based approach. The constant coefficient problems are considered as a special case, and the unconditional stability of compact schemes for such case is proved theoretically. The condition number of the amplification matrix is also analyzed, and an estimate for the same is derived. The examples are provided to support the assumption taken to assure stability.
\end{abstract}
\maketitle
\section{Introduction}\label{sec:intro}

\par The importance of convection-diffusion equations in the mathematical modeling of various physical phenomena is well known. These physical problems include option pricing problems in the stock market \cite{BlaS73}, computational fluid dynamics \cite{roach1976computational}, and various other applications \cite{isenberg1973heat, parlange1980water}. The development of numerical methods is inevitable to solve these equations due to the absence of analytical solutions in most cases. Various numerical methods are available in the literature to solve these partial differential equations (PDEs) \cite{AchP05, Can98, MKAW20}. The second-order accurate finite difference method (FDM) is one of them. The reason for the popularity of the FDM lies in its effortless implementation and showing desirable properties (consistency, stability, convergence, etc.) \cite{Tho95, TRan00, Jcstrik04}. Indeed, the FDM with extended computational stencil may also result in high-order accuracy, however, it complicates the implementation of boundary conditions. Moreover, the stability analysis in those cases becomes challenging. To overcome these difficulties, high-order accurate difference schemes were developed using compact stencils \cite{Rigal94, Spotz95, Vivek09}, commonly known as compact schemes.

\par The stability of a numerical method is desirable before its implementation to solve an initial or initial-boundary value problem. The compact schemes and their stability have been analyzed 
for different types of PDEs in the literature using different approaches. For example, in \cite{wang2018stability}, by considering a high-order compact alternating direction implicit (ADI) method for solving the three-dimensional unsteady convection–diffusion equation, its stability is demonstrated using the energy method. The stability of a compact scheme, based on the Hermitian formulation for constant coefficients parabolic problems is confirmed in \cite{wu2018stability,li2022stability}. 
While the former analyzed the spectral radius algebraically, the latter analyzed it by visualizing the eigenspectrum of the discretized operators. By similar visualization approach, the stability of the central compact scheme, employed for some first-order hyperbolic PDE, is investigated in \cite{vilar2015development}. Another instance of nonlinear equations appears in \cite{zhang2021convergence} where the Benjamin–Bona–Mahony equation is considered. In that, a compact finite difference scheme is utilized, and its stability aspect is discussed using the energy method. Needless to mention, various aspects of stability analysis of compact schemes extend beyond the class of classical PDEs. For example, such studies for fractional differential equations may be found in  \cite{li2018convergence,zhang2022pointwise}.

The compact scheme obtained via an auxiliary equation-based approach is known to be more efficient. We adopt that in this paper. So far the implementation of this class of compact schemes is concerned, various studies have already been reported for constant-coefficient convection-diffusion equations, where stability has been proved using the von Neumann method, see \cite{KPMA17, KPMD17, DurH15} and references therein. Moreover, compact schemes have also been developed for variable coefficient PDEs \cite{kalita2002class, Spotz01} in the literature. However, the stability analysis in \cite{kalita2002class} is performed using von Neumann approach by considering the coefficients as constants, while \cite{Spotz01} has not investigated the stability of the scheme. Note that, 
the applicability of the von Neumann stability analysis is restricted to the linear constant coefficient finite difference schemes (pp. 166 \cite{Tref96}). However, the investigation of spectral radius using matrix method for proving the stability of a numerical scheme 
is valid for variable coefficient problems also. The matrix method stability analysis becomes cumbersome for compact schemes, even in the constant coefficient case, as the amplification matrix cannot be expressed using a single tridiagonal matrix, unlike the case of FDM. Therefore, the common approach of establishing stability by studying the eigenvalues of the amplification matrix becomes challenging. Due to the intractability of the entries of the amplification matrix, locating the eigenvalues using theorems like Gershgorin's is not feasible. However, an alternative approach based on a difference equation has been proposed in \cite{mohebbi2010high} for a compact scheme for constant-coefficient convection-diffusion equations in one dimension. The difference equation is obtained from the eigenvalue problem of the amplification matrix. However, it is relevant to note that no progress has been made for multi-dimensional problems in this context. One dimensional problems with variable coefficients are even unaddressed in the existing literature.

In this paper, we extend the approach appearing in \cite{mohebbi2010high} to the variable coefficient generalization. The said extension is not straightforward and requires substantial novelty. Indeed the extension involves the derivation of a formula for the solution of the difference equation using mathematical induction. Thereby, despite the intractability of the entries of the amplification matrix, we succeeded in obtaining an expression of its characteristic polynomial. Using that polynomial, we derive a sufficient condition for unconditional stability. As per our knowledge, the present paper is the first instance of the stability analysis of a fully discrete compact scheme for a variable-coefficient convection-diffusion equation in the literature. 

In addition to the unconditional stability of the scheme, 
a highly accurate practical implementation needs good precision in the computation of the amplification matrix too. The related numerical error can be controlled by controlling the condition number of the matrix which needs to be inverted to get the amplification matrix. For this purpose, an understanding of the upper bound of the condition number becomes important. In this paper, the relevant condition number is shown to be of order $\mathcal{O} \left(\frac{\delta v}{{\delta z}^2}\right)$, where $\delta v$ and $\delta z$ are time and space step sizes respectively. The study of condition number related to the compact scheme is however less common in the literature. We refer to \cite{GoswamiKS} where a similar bound has been obtained in a different context. As per our knowledge, the derivation approach, adopted in the present paper, is a new addition to the literature.

This paper deviates from \cite{mohebbi2010high} in another crucial aspect. Here, the compact scheme, following \cite{Spotz01}, is considered to discretize the spatial dimension of the variable coefficient convection-diffusion equation. Further, for temporal discretization, the backward Euler, and Crank-Nicolson methods are considered which are different from the discretization techniques appearing in \cite{mohebbi2010high}. In \cite{mohebbi2010high}, the compact schemes are used for space discretization, and the cubic $C^1-$spline collocation method is applied for temporal discretization. 

\par The remainder of the paper is organized as follows: The compact scheme for variable coefficient convection-diffusion equation is presented in Sec. \ref{sec:model_problem} followed by its stability analysis. The stability of compact scheme for constant coefficient problem is proved in Sec. \ref{sec:stability}. The condition number and numerical results are given in Sec. \ref{sec:condition} and \ref{sec:numerical} respectively. The conclusion along with the future scope of proposed work are discussed in Sec.\ref{sec:conclu}.
\section{Stability analysis for variable coefficient problem}\label{sec:model_problem}
\noindent Let $a(z)$ and $b(z)$ be two bounded and twice continuously differentiable real-valued functions defined on finite open interval $(z_l, z_r),$ where $b(z)>\epsilon>0$,  $\: \forall \: z\in(z_l, z_r)$. Then, the variable coefficient convection-diffusion equation on $(z_l, z_r)$ can be written as follows:
\begin{equation}
\frac{\partial u}{\partial v}(v,z)+a(z)\frac{\partial u}{\partial z}(v,z)-b(z)\frac{\partial^2 u}{\partial z^2}(v,z)=0,\label{eq:oned_main_equation_2_v}
\end{equation}
where $z \in (z_l, z_r)$, and $0 \leq v \leq T,$ for some positive constant $T$. Moreover, we associate the following initial and boundary conditions with (\ref{eq:oned_main_equation_2_v})
\begin{align}
u(0,z)&=k(z), \quad z \in (z_l,z_r)\label{eq:initial_condition_1d_v},\\
u(v,z_l)&=h_1(v), \quad u(v,z_r)=h_2(v),\quad v> 0, \label{eq:oned_boun_v}
\end{align}
under the assumptions that $h_1$, and $h_2$ are smooth functions. Further, $h_1(0)=k(z_l)$, and $h_2(0)=k(z_r)$. Now, the compact scheme for variable coefficient convection-diffusion equation (\ref{eq:oned_main_equation_2_v}) is presented by following \cite{spotz1995high}. Consider the steady-state form of (\ref{eq:oned_main_equation_2_v}) as follows:
\begin{align}
    \label{Steady-state_v}
    a(z)\frac{\partial u}{\partial z}-b(z)\frac{\partial^2 u}{\partial z^2}=0.
\end{align}
The uniform grids are considered. Indeed we discretize $z$ as $z_i = z_l + i\delta z$, where $0 \leq i \leq N$  with constant step size $\delta z$. Here the natural number $N$ denotes the mesh size.
Standard central difference approximations are considered to discretize (\ref{Steady-state_v}) as follows:
\begin{equation}
\label{Truncation_error}
    a_i\Delta_{z} u_{i}-b_i\Delta_{zz} u_{i}-\tau_{i}=0,
\end{equation}
where $\Delta_{z}u_i$, and $\Delta_{zz}u_i$ are central difference operators in space of first and second order respectively. Here $u_i=u(z_i),\:a_i=a(z_i)$ and $b_i=b(z_i)$. Moreover, the truncation error $\tau_i$ is given by 
\begin{equation}
\label{tau_q_v}
\tau_i=\left[\frac{{\delta z}^2}{12}\left(2 a\frac{\partial^3 u}{\partial z^3}- b \frac{\partial^4 u}{\partial z^4}\right)(z_i)\right]_i+\mathcal{O}(\delta z^4).
\end{equation}
In order to accomplish fourth-order accuracy, let us take (\ref{Steady-state_v}) as an auxiliary equation and approximate $\frac{\partial^3 u}{\partial z^3}$ and $\frac{\partial^4 u}{\partial z^4}$ as follows:
$$
\begin{aligned}
\left.\frac{\partial^3 u}{\partial z^3}\right\vert_i =& \left[-\frac{1}{b_i}\Delta_{z}b_i\Delta_{zz}u_i+\frac{a_i}{b_i}\Delta_{zz}u_i+\frac{1}{b_i}\Delta_{z}a_i\Delta_{z}u_i\right]+\mathcal{O}({\delta z}^2),\\
\left.-b \frac{\partial^4 u}{\partial z^4}\right\vert_i =& \left[-\frac{2( \Delta_{z} b_i)^2}{b_i}+\frac{3a_i}{b_i}{\Delta_{z}b_i}-\frac{{a_i}^2}{b_i}+\Delta_{zz}b_i-2 \Delta_{z}a_i\right] \Delta_{zz}u_i\\
&+\left[(2\Delta_{z}b_i-a_i)\frac{\Delta_{z}a_i}{b_i}-\Delta_{zz}a_i\right]\Delta_{z}u_i+\mathcal{O}({\delta z}^2).
\end{aligned}
$$
Using the above expressions and relations (\ref{Truncation_error})-(\ref{tau_q_v}), fourth-order accurate approximation of steady state equation (\ref{Steady-state_v}) is given as follows:
\begin{align}
    \label{Steady_state_discr}
    \zeta_i\Delta_{z}u_i-\alpha_i \Delta_{zz}u_i+\mathcal{O}({\delta z}^4)=0,
\end{align}
where 
$$
\begin{aligned}
  \zeta_i &=\left[a_i-\frac{\delta {z^2}}{12}\left(\frac{a_i}{b_i}\Delta _z a_i+\frac{2}{b_i}\Delta_z a_i\Delta_z b_i- \Delta_{zz} a_i\right)\right],\\
  \alpha_i&=\left[ b_i+\frac{\delta {z^2}}{12}\left(\frac{a_i}{b_i}\Delta_z b_i-2\Delta_z a_i+\Delta_{zz} b_i-\frac{2}{b_i}(\Delta_z b_i)^2+\frac{a_i^2}{b_i}\right)\right].
\end{aligned}
$$
Let $u_{i}^{m}$ denote the solution of (\ref{eq:oned_main_equation_2_v}) at time grid point $v^{m}= m \delta v, m = 0, 1, \dots , M$ and space mesh point $z_i$, and $\Delta_{v}^{+}u^m_i=\frac{u^{m+1}_i-u^m_i}{\delta v}$ for constant time step $\delta v$. Using relation (\ref{Steady_state_discr}), the fully-discrete problem for the unsteady equation (\ref{eq:oned_main_equation_2_v}) can be written as
\begin{align}
    \left[1-\frac{{\delta z}^2}{12}\left(\frac{a_i}{b_i}+\frac{2\Delta_{z}b_i}{b_i}\right)\Delta_{z}+\frac{{\delta z}^2}{12}\Delta_{zz}\right]\Delta_{v}^{+}u_{i}^{m} +\zeta_i\Delta_{z}u_{i}^{m}-\alpha_i \Delta_{zz}u_{i}^{m}+\mathcal{O}({\delta z}^4)=0.
\end{align}
The expressions for $\Delta_{z}u_{i}^{m}$ and $\Delta_{zz}u_{i}^{m}$ are given as
\begin{equation}\label{eq:firstf_2_v}
 \Delta_{z}u^m_i=\frac{u^m_{i+1}-u^m_{i-1}}{2\delta z},\:\: \Delta_{zz}u^m_{i}=\frac{u^m_{i+1}-2u^m_{i}+u^m_{i-1}}{{\delta z}^2}.
\end{equation} 
Let $\gamma_i =\left( \frac{a_i}{b_i}+\frac{2}{b_i}\Delta_z b_i\right)$, then fully discrete problem for (\ref{eq:oned_main_equation_2_v}) can be written as follows: 
\begin{align}\label{eq:oned_fully_discrete_compact_main_v}
\Delta_v^+ u_i^m&-\frac{{\delta z}^2}{12}\left(\gamma_i\Delta_z\Delta_v^+ u_i^m-\Delta_{zz}\Delta_v^+ u_i^m\right)+(1-\theta)\left(\zeta_i\Delta_z u_i^m-\alpha_i\Delta_{zz} u_i^m\right)\nonumber \\
&+\theta \left(\zeta_i\Delta_z u_i^{m+1}-\alpha_i\Delta_{zz} u_i^{m+1}\right)=\mathcal{O}(\{\delta v\}^{\varepsilon(\theta)}, \{\delta z\}^4),
\end{align}
for $1 \leq i \leq N-1,$ and $m=0,1, \dots,M-1$. The order $\varepsilon(\theta)$ in the error term on the right side of (\ref{eq:oned_fully_discrete_compact_main_v}) depends on $\theta$. The Crank-Nicolson, and backward Euler time discretization appears for $\theta=\frac{1}{2}$, and $1$ respectively. The order of accuracy in time is $\varepsilon(\theta)=2$ for $\theta = 1/2$, and $\varepsilon(\theta)=1$ for $\theta =1$. 
If $U^m_i$ denotes the approximate value of $u^m_i$, then (\ref{eq:oned_fully_discrete_compact_main_v}) can be rewritten for $1 \leq i \leq N-1,$ and $m\geq0$ as follows:
\begin{align}\label{eq:oned_fully_discrete_compact1_v}
\Delta_v^+ U_i^m&-\frac{{\delta z}^2}{12}\left(\gamma_i\Delta_z\Delta_v^+ U_i^m-\Delta_{zz}\Delta_v^+ U_i^m\right)+(1-\theta)\left(\zeta_i\Delta_z U_i^m-\alpha_i\Delta_{zz} U_i^m\right)\nonumber \\
&+\theta \left(\zeta_i\Delta_z U_i^{m+1}-\alpha_i\Delta_{zz} U_i^{m+1}\right)=0,
\end{align}
with $U^m_0=h_1(v_m)$, and $U^m_N=h_2(v_m)$, for all $m\geq0$. Next, we wish to rewrite (\ref{eq:oned_fully_discrete_compact1_v}) by plugging in the expressions of difference operators, appearing in (\ref{eq:firstf_2_v}). To this end, we denote the column vector $U^m=[U^m_1,U^m_2,...,U^m_{N-1}],$ $\forall$ $m \geq 0$, and define the constants depending on $\delta z$ and $\delta v$ as follows:
\begin{align}
  p_{i}=\frac{2+\delta z \gamma_i}{24\delta v}, \: q_{i}=\frac{5}{6 \delta v}, \: \text{and} \: r_i=\frac{2-\delta z \gamma_i}{24\delta v}, \:\forall\: i=1,2,\dots,N-1.\:\label{eq:c_v}
\end{align}
\begin{align}
l_{i}=-\frac{\zeta_i}{2\delta z}-\frac{\alpha_i}{{\delta z}^2}, \: m_{i}=\frac{2\alpha_i}{{\delta z}^2},\: \text{and} \: 
n_i=\frac{\zeta_i}{2\delta z}-\frac{\alpha_i}{{\delta z}^2},\:\forall \:i=1,2,\dots,N-1.\label{eq:y_v}
\end{align}
Then Eq. (\ref{eq:oned_fully_discrete_compact1_v}) can be written as 
\begin{equation}\label{eq:oned_fully_discrete_compact_v}
[(1-\theta)X+\theta(X+Y)]U^{m+1}=[\theta X+ (1-\theta)(X-Y)]U^m+F^m,
\end{equation}
where $X$ and $Y$ are the following tridiagonal matrices of order $(N-1)\times(N-1)$:
\begin{equation}
\label{eq_matrix_x_v}
X\vcentcolon =\left[ 
\begin{array}{c c c c c c c c c} 
  q_1 & r_1 & & & \cdots & & & &\\ 
  p_2 & q_2 & r_2 & & \cdots & & & &\\ & p_3 & q_3 & r_3 &\cdots & & & & \\& & & & \vdots & & & &\\
 & & & & \cdots & p_{N-3} & q_{N-3} & r_{N-3} & \\ 
& & & & \cdots & & p_{N-2} & q_{N-2} & r_{N-2} \\ 
 & & & & \cdots & & & p_{N-1} & q_{N-1}
\end{array} 
\right], \quad
\end{equation}
\begin{equation}
\label{eq_matrix_y_v}
Y\vcentcolon =\left[ 
\begin{array}{c c c c c c c c c} 
 m_1 & n_1 & & & \cdots & & & &\\ 
 l_2 & m_2 & n_2 & & \cdots & & & &\\ 
 & l_3 & m_3 & n_3 &\cdots & & & &\\ 
 & & & & \vdots & & & &\\
 & & & & \cdots & l_{N-3} & m_{N-3} & n_{N-3} & \\ 
 & & & & \cdots & & l_{N-2} & m_{N-2} & n_{N-2} \\ 
 & & & & \cdots & & & l_{N-1} & m_{N-1}
\end{array} 
\right],
\end{equation}
and $F^m$ is a $N-1$ dimensional column vector consists of boundary values. Its expression is given for each $\theta$ in the following subsections. 
\subsection{Backward Euler compact scheme}\label{ssec:backward_euler}
Taking $\theta=1$ in (\ref{eq:oned_fully_discrete_compact_v}), the fully discrete backward Euler compact scheme can be written as:
\begin{equation}\label{eq:xy_B_v}
(X+Y) U^{m+1}=XU^m+F^m, 
\end{equation}
where $X$ and $Y$ are given in (\ref{eq_matrix_x_v}), and (\ref{eq_matrix_y_v}) respectively. Moreover 
$$F^m=[p_1h_1(v^m)-(p_1+l_{1})h_1(v^{m+1}),0,\dots,0,r_{N-1}h_2(v^m)-(r_{N-1}+n_{N-1})h_2(v^{m+1})].$$
$X$ is diagonally dominant for sufficiently small $\delta z$ due to the boundedness of $a$, $b$ and $b^{-1}$ on the domain. Hence $X$ is invertible. Using $W=X^{-1}Y$ in (\ref{eq:xy_B_v}), we get
\begin{equation}\label{eq:w_B_v}
(I+W)U^{m+1}=U^{m}+X^{-1}F^m.
\end{equation}
The stability of (\ref{eq:w_B_v}), 
is asserted in  Theorem \ref{theorem:stability_B_v}. To establish that we need to study the eigenvalues of $W$.
\begin{remark}
To locate the eigenvalues of matrix $W$, we aim to derive the characteristic polynomial. Note that the entries of $W$ are not tractable since $W$ involves $X^{-1}$. Hence the determinant approach for characteristic polynomial is not applicable. Using the tridiagonal structure of matrices $X$ and $Y$, we propose a difference equation approach for deriving the characteristic polynomial of $W$.
\end{remark}
\noindent 
Let $\beta=(\beta_{1}, \beta_{2}, \ldots, \beta_{N-1})\in \mathbb{C}^{N-1}$ be an eigenvector of $W$ corresponding to an eigenvalue $\lambda\in \mathbb{C}$. So we have $X^{-1} Y \beta=\lambda \beta$ or $(Y - \lambda X) \beta =0$. This linear system in $\beta$ can be written explicitly as
\begin{align}\label{eq:sys_eq}
\left.\begin{array}{rl}
\left(m_1-\lambda q_1\right) \beta_1+\left(n_1-\lambda r_1\right)\beta_2&=0,\\
\left(l_j-\lambda p_j\right)\beta_{j-1}+\left(m_j-\lambda q_j\right) \beta_j+\left(n_j-\lambda r_j\right)\beta_{j+1}&=0,\quad
j=2, \ldots, N-2,\\
\left(l_{N-1}-\lambda p_{N-1}\right)\beta_{N-2}+\left(m_{N-1}-\lambda q_{N-1}\right) \beta_{N-1}&=0.
\end{array}\right\}
\end{align}
Then the above system can be rewritten as
\begin{equation}\label{linear_system_FB_v}
C_j\beta_{j-1}+B_j \beta_j+A_j\beta_{j+1}=0,\:\forall\:j=1,2,\dots,N-1,
\end{equation}
with following constraints on newly added extra variables $\beta_0$, $\beta_N$ as
\begin{align}\label{constraint}
\beta_{0}=\beta_{N}=0,
\end{align}
where for $j=1,\ldots, N-1$
\begin{equation}\label{eq:ajbjcj}
C_j=l_j-\lambda p_j,\:B_j=m_j-\lambda q_j, \: \text{and}\: A_{j}=n_{j}-\lambda r_{j}.
\end{equation}

\noindent Here we present an expression of the general solution of \eqref{linear_system_FB_v} in the following Lemma.
\begin{lemma}\label{lem2}
For any fixed $\lambda\in \mathbb{C}$, and $N\ge 2$, let the vector $x= (x_0,x_1,\dots,x_N)\in \mathbb{C}^{N+1}$ solve \eqref{linear_system_FB_v} with $x_0=1$. Then
\begin{align}\label{lemeq}
\left( \prod_{j=1}^{N-1}A_j \right) x_N=\sum_{s\in\mathscr{A}_N}s_0(s_1,x_1)\prod_{j=1}^{N-1}s_j
\end{align}
where $\mathscr{A}_N$ is a collection of finite sequences $s:=\{s_j\}_{j=1}^{N-1}$ such that
\begin{align}\label{eq:A}
\nonumber 
s_{N-1} & \in \left\{-B_{N-1}, -C_{N-1}\right\}\\
s_{j-1} & \in \left\{
\begin{array}{ll}
\{A_{j-1}\} &\textrm{ if } s_j=-C_j,\\
\{-B_{j-1}, -C_{j-1}\} &\textrm{ if }  s_j\in\{A_j,-B_j\}
\end{array}\right. \forall \:\:  2\le j < N,
\end{align}
and 
$$ s_0(s_1,x_1) := \left\{
\begin{array}{ll}
1 &\textrm{ if } s_1=-C_1,\\
x_1 &\textrm{ if }  s_1\in\{A_1,-B_1\}.
\end{array}\right. 
$$
\end{lemma}
\begin{proof}
Set $N=2$. Then left side of \eqref{lemeq} is $A_1x_2$. Furthermore, $\mathscr{A}_2= \{ \{s_1\} : s_1 \in \{-B_1, -C_1\}\}$. Thus using the definition  of $s_0(s_1,x_1)$, the right side of \eqref{lemeq} is 
$$\sum_{s\in\mathscr{A}_2}s_0(s_1,x_1)s_1 = x_1.(-B_1) + 1.(-C_1).$$
Also \eqref{linear_system_FB_v} and $x_0=1$ give $A_1x_2=-B_1x_1 -C_1$. Hence, both sides of \eqref{lemeq} are equal for $N=2$. Next we show that \eqref{lemeq} holds for $N=N'+1$ by assuming that \eqref{lemeq} holds for $N=2,\ldots, N'$.
From \eqref{linear_system_FB_v} and the induction hypothesis we have 
\begin{align}\label{iteration}
\nonumber A_{N'} x_{N'+1} =& -C_{N'} x_{N'-1} -B_{N'} x_{N'} \\
\left(\prod_{j=1}^{N'}A_j\right) x_{N'+1}=&-C_{N'} A_{N'-1} \left(\prod_{j=1}^{N'-2}A_j \right) x_{N'-1} - B_{N'} \left(\prod_{j=1}^{N'-1}A_j \right) x_{N'}\\
\nonumber =&  -C_{N'} A_{N'-1} 
\sum_{s\in\mathscr{A}_{N'-1}}s_0(s_1,x_1)\prod_{j=1}^{N'-2}s_j -  B_{N'} \sum_{s\in\mathscr{A}_{N'}}s_0(s_1,x_1)\prod_{j=1}^{N'-1}s_j\\
\nonumber =& \sum_{ \{s\in \mathscr{A}_{N'+1} : s_{N'}= -C_{N'} \}} s_0(s_1,x_1)\prod_{j=1}^{N'}s_j + \sum_{ \{s\in \mathscr{A}_{N'+1}: s_{N'}=-  B_{N'}\}} s_0(s_1,x_1)\prod_{j=1}^{N'}s_j\\
\nonumber =& \sum_{s\in \mathscr{A}_{N'+1}} s_0(s_1,x_1)\prod_{j=1}^{N'}s_j.
\end{align}
Thus the proof is complete.
\end{proof}

\noindent We introduce the following notations
\begin{equation}\label{eq:d1d2}
D_N^{1}=\sum_{\left\{s\in\mathscr{A}_N|s_1\neq -C_1\right\}}\prod_{j=1}^{N-1}s_j,\quad D_N^{2}=\sum_{\left\{s\in\mathscr{A}_N|s_1= -C_1\right\}}\prod_{j=1}^{N-1}s_j.
\end{equation}
Since $s_j$ is a first degree polynomial in $\lambda$ for each $j$, $D_{N}^{1}$  and $D_{N}^{2}$ are polynomials in $\lambda$ of degree atmost $N-1$. Using the above notation, (\ref{lemeq}) may be rewritten as 

\begin{align}\label{lemeq1}
\left( \prod_{j=1}^{N-1}A_j \right) x_N= D^1_N x_1+D^2_N.
\end{align}

\begin{theorem}\label{theorem:real part_v}
Let $\lambda$ be an eigenvalue of matrix $W$, then $\lambda$ solves $D_{N}^{1}=0$, where $D_N^1$ is given in (\ref{eq:d1d2}).
\end{theorem}
\begin{proof}
For any fixed $\lambda\in\mathbb{C}$, the general solutions of (\ref{linear_system_FB_v}) for the unknown vector $(\beta_0,\beta_1,\dots,\beta_N)$ form a linear space of dimension at least $2$ (Theorem 2.21 \cite{SElaydi05}),  when no constraints like \eqref{constraint} are imposed. If $\beta^{'}$ and $\beta^{''}$ denote two independent solutions to (\ref{linear_system_FB_v}), then a linear combination 
\begin{equation}\label{betaj}
\beta_{j}=K_1\beta_{j}^{'}+K_2\beta_{j}^{''},\:\:\forall\:j=0,1,\dots,N,
\end{equation}
also solves (\ref{linear_system_FB_v}) for any pair of scalars $K_1,K_2$. Due to the independence of $\beta'$ and $\beta''$, without loss of generality we may take 
\begin{equation}\label{eq:beta_prime}
\beta_{0}^{'}=\beta_{0}^{''}=1 \quad \text{ and} \quad \beta_{1}^{'} \neq \beta_{1}^{''}.
\end{equation}
Thus none of $\beta^{'}$ and $\beta^{''}$ satisfy \eqref{constraint}. Hence they do not solve \eqref{eq:sys_eq}. However, a linear combination of $\beta^{'}$ and $\beta^{''}$ for a specific value of $\lambda$ may satisfy \eqref{eq:sys_eq}. Indeed, we would find a necessary condition on $\lambda$ so that there is a linear combination of  $\beta^{'}$ and $\beta^{''}$  which solves \eqref{eq:sys_eq}.

\noindent Since $\beta'$ and $\beta''$ satisfy
\eqref{linear_system_FB_v} and $\beta'_0=\beta''_0=1$, Lemma \ref{lem2} is applicable. Hence from (\ref{lemeq1}), we get

\begin{equation}\label{eq:betannprime}
\left(\prod_{j=1}^{N-1}A_j\right)\beta_N^{'}=D^1_N\beta'_1+D^2_N, \quad \quad \left(\prod_{j=1}^{N-1}A_j\right)\beta_N^{''}=D^1_N\beta''_1+D^2_N.
\end{equation}

\noindent Now we recall that if $\lambda\in \mathbb{C}$ is an eigenvalue of $W$, then there is a non-trivial solution $(\beta_0, \ldots, \beta_N)$ of the system \eqref{linear_system_FB_v}-(\ref{constraint}) having  the form (\ref{betaj}). The condition $\beta_0=0$ in (\ref{betaj}) implies $K_1=-K_2$. Additionally using $\beta_N=0$, (\ref{betaj}) gives $K_1(\beta_{N}^{'}-\beta_{N}^{''}) =0$, i.e., $\beta_{N}^{'}=\beta_{N}^{''}$. Consequently from (\ref{eq:betannprime}), we get $D_{N}^{1}\beta_1^{'}+D_{N}^{2}=D_{N}^{1}\beta_1^{''}+D_{N}^{2}.$ Hence $D_N^{1}(\beta_1^{'}-\beta_1^{''})=0.$ Thus, using (\ref{eq:beta_prime}), we have $D_{N}^{1}=0$. Thus $\lambda$ solves $D_{N}^{1}=0$. 
\end{proof}


\begin{remark} An algorithm may easily be developed for symbolic computation of $\left(\prod_{j=1}^{N-1}A_j\right)\beta^{'}_N$ for general $N$, from the recurrence relation (\ref{linear_system_FB_v}). Then using (\ref{eq:betannprime}), the coefficient of $\beta'_1$, i.e. $D^1_N$, can be obtained and its roots can be computed subsequently. To illustrate this, three MATLAB algorithms are provided in Appendix A. There, the algorithm in Section \ref{ssec:1} computes the symbolic expressions for $D^1_N$. Further, the function in Section \ref{ssec:2} gives the symbolic/numerical values for the roots of $D^1_N$. Finally, both of these functions are called in another algorithm in Section \ref{ssec:3}. The algorithm in Section \ref{ssec:3} makes use of the following theorem (Theorem \ref{theorem:stability_B_v}) to conclude the stability of the backward Euler compact scheme (\ref{eq:w_B_v}) under the specific choices of parameters entered by a user.
\end{remark}

\begin{theorem}\label{theorem:stability_B_v}
If the real part of each root of polynomial $D^1_N$ is positive, then the backward Euler compact scheme (\ref{eq:w_B_v}) is stable. 
\end{theorem}
\begin{proof}
From Theorem \ref{theorem:real part_v}, the eigenvalues of $W$ are the roots of $D_N^1$. Rewriting the equation \eqref{eq:w_B_v} as $$U^{m+1}=HU^{m}+(X+Y)^{-1}F^m,$$ where $H=(I+W)^{-1}$.
If $\lambda$ and $\beta$ is a pair of eigenvalue and eigenvector of $W$, we have $
H\beta=\frac{1}{1+\lambda}\beta$. Consequently, $\frac{1}{1+\lambda}$ is an eigenvalue of $H$. Hence, backward Euler compact scheme (\ref{eq:w_B_v}) is stable if we have 
$\frac{1}{|1+\lambda|} <1$ for each $\lambda$. The above inequality follows as real part of each $\lambda$ is assumed to be positive.
\end{proof}

\begin{remark}
Note that, the assertion in Theorem \ref{theorem:stability_B_v} holds under certain assumptions. 
In Table \ref{tab:eigen_test}, we showed with specific choices of parameters that all the roots of $D^1_N$ have a positive real part, hence it is not an unrealistic assumption.
\end{remark}

\begin{remark}
It is interesting to note that the number of elements in $\mathscr{A}_N$ is $F_{N+1}$, where $F_N$ is the $N^{th}$ Fibonacci number. Using a rough upper bound on the growth of the Fibonacci series, the growth of complexity of computing $D^1_N$, directly from \eqref{eq:d1d2} can be estimated. That estimated complexity turns out to be dominated by $\mathcal{O}(N2^N)$. On the other hand, the other approach of deriving a characteristic polynomial involves the computation of the determinant, having a complexity of order factorial of $N$, according to the Laplace expansion. As $N!$ grows faster than  $N2^N$, and dominates for $N\ge 6$, computation of $D^1_N$  from \eqref{eq:d1d2} appears advantageous. However, there are smarter algorithms for computing determinants with cubic complexity. Similarly, an iterative use of \eqref{iteration} yields a scheme of complexity $\mathcal{O}(N)$ for symbolic computation of $\left(\prod_{j=1}^{N-1}A_j\right)x_N$. Subsequently, a use of \eqref{lemeq1} determines the expression of $D^1_N$.
\end{remark}
\subsection{Crank-Nicolson compact scheme}\label{ssec:cranknicolson}
Let us define $\tilde{Y}=\frac{1}{2}Y$, and thus $\tilde{W}:=X^{-1}\tilde{Y}=\frac{1}{2}W.$ 
Taking $\theta=1/2,$ in (\ref{eq:oned_fully_discrete_compact_v}) the fully discrete Crank-Nicolson compact scheme can be written as
\begin{equation}\label{eq:xy_CN_v}
(X+\tilde{Y})U^{m+1}=(X-\tilde{Y})U^m+\tilde{F}^m, 
\end{equation}
where 
\begin{align*}
\tilde{F}^m =\left[\left(p_1-\frac{l_{1}}{2}\right)h_1(v_m)-\left(p_1+\frac{l_{1}}{2}\right)h_1(v_{m+1}),0,...,0,\left(r_{N-1}-\frac{n_{N-1}}{2}\right)h_2(v_m)\right.\\
\left.-\left(r_{N-1}+\frac{n_{N-1}}{2}\right)h_2(v_{m+1})\right]\nonumber.
\end{align*}
\noindent Furthermore, (\ref{eq:xy_CN_v}) can be written as:
 \begin{equation}\label{eq:w_CN_v}
(I+\tilde{W})U^{m+1}=(I-\tilde{W})U^{m}+X^{-1}\tilde{F}^m.
\end{equation}

\begin{theorem}\label{theorem:stability_v}
If the real part of each root of polynomial $D^1_N$ is positive, then the Crank-Nicolson compact scheme (\ref{eq:w_CN_v}) is stable.
\end{theorem}
\begin{proof}
Let $\tilde{\lambda}$ be an eigenvalue of matrix $\tilde{W}$, then using $2\tilde{W} = W$, and Theorem \ref{theorem:real part_v}, $2\tilde{\lambda}$ is a root of the polynomial $D^1_N$. Hence, if every root of $D^1_N$ has positive real part, the real part of $\tilde{\lambda}$ is also positive. Therefore $\tilde{W}$ is invertible, and we can rewrite the Eq. \eqref{eq:w_CN_v} as
$$U^{m+1}=\tilde{H}U^{m}+(X+\tilde{Y})^{-1}\tilde{F}^m,$$ 
where $\tilde{H}=(I+\tilde{W})^{-1}(I-\tilde{W})$.  
Moreover, $\frac{1-\tilde{\lambda}}{1+\tilde{\lambda}}$ is an eigenvalue of $\tilde{H}$ iff $\tilde{\lambda}$ is an eigenvalue of $\tilde{W}$. Thus, Crank-Nicolson compact scheme (\ref{eq:w_CN_v}) is stable if we have $\abs{\frac{1-\tilde{\lambda}}{1+\tilde{\lambda}}}<1$ for each $\tilde{\lambda}$. Since real part of $\tilde{\lambda}$ is positive, therefore $\abs{1-\tilde{\lambda}}<\abs{1+\tilde{\lambda}}$. Thus $\abs{\frac{1-\tilde{\lambda}} {1+\tilde{\lambda}}}<1,$ as desired.
\end{proof}

\section{Special Case: Constant coefficient problem}\label{sec:stability} 
\noindent In this section, we consider constant coefficient version of convection diffusion equation (\ref{eq:oned_main_equation_2_v}) by replacing $a(z)$ and $b(z)$ by scalars $a$ and $b$ respectively. We prove below that the assumption on the characteristic polynomial $D^1_N$ in Theorems \ref{theorem:stability_B_v} and \ref{theorem:stability_v} are true for this special case. 
Consider the following constant coefficient convection-diffusion equation
\begin{equation}
\frac{\partial \psi}{\partial v}(v,x)+a\frac{\partial \psi}{\partial x}(v,x)-b\frac{\partial^2 \psi}{\partial x^2}(v,x)=0,\label{eq:oned_main_equation1}
\end{equation}
If we take $u = \dfrac{b}{a}\psi,$ and $z=\dfrac{a}{b}x$ in above equation (\ref{eq:oned_main_equation1}), we have
\begin{equation}
\frac{\partial u}{\partial v}(v,z)+c\frac{\partial u}{\partial z}(v,z)-c\frac{\partial^2 u}{\partial z^2}(v,z)=0.\label{eq:oned_main_equation_2}
\end{equation}
\noindent Using (\ref{eq:oned_fully_discrete_compact1_v}), the compact scheme for equation (\ref{eq:oned_main_equation_2}) is as follows:
\begin{align}\label{eq:oned_fully_discrete_compact1}
\Delta_v^+ U_i^m&-\frac{{\delta z}^2}{12}\left(\Delta_z\Delta_v^+ U_i^m-\Delta_{zz}\Delta_v^+ U_i^m\right)+(1-\theta)c\left(\Delta_z U_i^m-\left(1+\frac{{\delta z}^2}{12}\right)\Delta_{zz} U_i^m\right)\nonumber \\
&+\theta c\left(\Delta_z U_i^{m+1}-\left(1+\frac{{\delta z}^2}{12}\right)\Delta_{zz} U_i^{m+1}\right)=0.
\end{align}
We define the constants depending on $\delta z$ and $\delta v$ as
\begin{align}
c_1=\frac{2+\delta z}{24\delta v},\: c_2&=\frac{5}{6 \delta v}, \: \text{and} \: c_3=\frac{2-\delta z}{24\delta v},\label{eq:c}\\
y_1=-\frac{c}{2 \delta z}-\frac{\left(c+\dfrac{c{\delta z}^2}{12}\right)}{{\delta z}^2},\: y_2&=2\frac{\left(c+\dfrac{c{\delta z}^2}{12}\right)}{{\delta z}^2}, \: \text{and} \: y_3=\frac{c}{2 \delta z}-\frac{\left(c+\dfrac{c{\delta z}^2}{12}\right)}{{\delta z}^2}.\label{eq:y}
\end{align}

\begin{remark}
For the constant coefficient case, a direct substitution of constant parameters in \eqref{Steady_state_discr} gives the following values $\gamma_i=1$, $\zeta_i=c$, and $\alpha_i=c+\dfrac{c{\delta z}^2}{12}$, for all $1\leq i\leq N-1$. Hence, the values of $p_i$, $q_i$, $r_i$, $l_i$, $m_i$, and $n_i$ coincide with those of $c_1$, $c_2$, $c_3$, $y_1$, $y_2$, and $y_3$ respectively. Thus the coefficients in (\ref{eq:c}) and (\ref{eq:y}) are special case of the coefficients obtained in (\ref{eq:c_v}) and (\ref{eq:y_v}). 
\end{remark}

\noindent Further, denote $\frac{\delta v}{{\delta z}^2}$ as $d$ and we write $$y_1=\frac{-c}{2 \delta v}\left[\left(2+\delta z\right)d+\frac{\delta v}{6}\right], 
\quad y_2=\frac{c}{ \delta v} \left(2d+\frac{\delta v}{6}\right), \quad \mbox{and} \quad y_3=\frac{-c}{2 \delta v}\left[\left(2-\delta z\right)d+\frac{\delta v}{6}\right].$$
Then Eq. (\ref{eq:oned_fully_discrete_compact1}) can be written as 
\begin{equation}\label{eq:oned_fully_discrete_compact}
[(1-\theta)X_1+\theta(X_1+Y_1)]U^{m+1}=[\theta X_1+ (1-\theta)(X_1-Y_1)]U^m+F_1^m,
\end{equation}
where $X_1$ and $Y_1$ are the following tridiagonal matrices of order $(N-1)\times(N-1)$:
\begin{equation}
\label{eq_matrix_x}
X_1\vcentcolon =\left[ 
\begin{array}{c c c c c c c c c} 
  c_2 & c_3 & & & \cdots & & & &\\ 
  c_1 & c_2 & c_3 & & \cdots & & & & \\ 
 & c_1 & c_2 & c_3 &\cdots & & & & \\ 
 & & & & \vdots & & & &\\
 & & & & \cdots & c_1 & c_2 & c_3 & \\ 
 & & & & \cdots & & c_1 & c_2 & c_3 \\ 
 & & & & \cdots & & & c_1 & c_2
\end{array} 
\right], \quad
\end{equation}
\begin{equation}
\label{eq_matrix_y}
Y_1\vcentcolon =\left[ 
\begin{array}{c c c c c c c c c} 
 y_2 & y_3 & & & \cdots & & & &\\ 
 y_1 & y_2 & y_3 & & \cdots & & & & \\ 
 & y_1 & y_2 & y_3 &\cdots & & & & \\ 
 & & & & \vdots & & & &\\
 & & & & \cdots & y_1 & y_2 & y_3 & \\ 
 & & & & \cdots & & y_1 & y_2 & y_3 \\ 
 & & & & \cdots & & & y_1 & y_2
\end{array} 
\right],
\end{equation}
and $F_1^m$ is a $N-1$ dimensional coloum vector consists of boundary values. Its expression is given for each $\theta$ in the following section. In the next section, the stability of the above fully discrete problem (\ref{eq:oned_fully_discrete_compact}) for backward Euler ($\theta = 1$) and Crank-Nicolson ($\theta =\frac{1}{2}$) schemes are proved with no additional assumption on the coefficient.
If the real constants $A,B$ and $C$ are such that $\frac{B^2}{4AC}\in [0,1]$, then clearly there is a unique $\phi\in [0,\pi]$ so that $\cos^2{\frac{\phi}{2}}= \frac{B^2}{4AC}$. A direct calculation gives that this identity is equivalent to the following relation
\begin{equation}
\label{eq:lemma_1}
\frac{(2B^2-4AC+2B\sqrt{B^2-4AC})}{4AC}=2\cos^2{\frac{\phi}{2}}-1+2\cos{\frac{\phi}{2}}\sqrt{-\sin^2{\frac{\phi}{2}}}=\cos{\phi}+i\sin{\phi},
\end{equation}
where $i=\sqrt{-1}$. The above result has also appeared in \cite{mohebbi2010high}. 
We will use this relation in the following sections.
\subsection{Backward Euler compact scheme} Taking $\theta=1$ in (\ref{eq:oned_fully_discrete_compact}), the fully discrete backward Euler compact scheme can be written as:
\begin{equation}\label{eq:xy_B}
(X_1+Y_1) U^{m+1}=X_1U^m+F_1^m, 
\end{equation}
where $X_1$ and $Y_1$ are given in (\ref{eq_matrix_x}), and (\ref{eq_matrix_y}) respectively. Moreover 
$$F_1^m=[c_1h_1(v^m)-(c_1+y_{1})h_1(v^{m+1}),0,\dots,0,c_3h_2(v^m)-(c_3+y_{3})h_2(v^{m+1})].$$
Evidently, $X_1$ is diagonally dominant, and hence invertible. Using $W_1=X_1^{-1}Y_1$ in (\ref{eq:xy_B}), we get
\begin{equation}\label{eq:w_B}
(I+W_1)U^{m+1}=U^{m}+X_1^{-1}F_1^m.
\end{equation}
The stability of (\ref{eq:w_B}), 
is asserted in Theorem \ref{lemma:stability_B}. To establish that 
we prove the following results.
\begin{theorem}\label{theorem:real part}
Let $\lambda$ 
be an eigenvalue 
of matrix $W_1$, then the real part of $\lambda$ is positive.
\end{theorem}
\begin{proof}
Let $\beta=(\beta_{1}, \beta_{2}, \ldots, \beta_{N-1})$ be  an eigenvector corresponding to the eigenvalue $\lambda$. So we have $X_1^{-1} Y_1 \beta=\lambda \beta$ or $(Y_1 - \lambda X_1) \beta =0$. This linear system in $\beta$ can be written explicitly as
\begin{align*}
\left.\begin{array}{rl}
\left(y_2-\lambda c_2\right) \beta_1+\left(y_3-\lambda c_3\right)\beta_2&=0,\\
\left(y_1-\lambda c_1\right)\beta_{j-1}+\left(y_2-\lambda c_2\right) \beta_j+\left(y_3-\lambda c_3\right)\beta_{j+1}&=0,\quad
j=2, \ldots, N-2,\\
\left(y_1-\lambda c_1\right)\beta_{N-2}+\left(y_2-\lambda c_2\right) \beta_{N-1}&=0.
\end{array}\right\}
\end{align*}
Let us introduce $\beta_{0}=\beta_{N}=0$. Then the above system can be rewritten as
\begin{equation}\label{linear_system_FB}
\left(y_1-\lambda c_1\right)\beta_{j-1}+\left(y_2-\lambda c_2\right) \beta_j+\left(y_3-\lambda c_3\right)\beta_{j+1}=0,
\end{equation}
for all $j=1,2,\dots,N-1$. Let $A=\left(y_3-\lambda c_3\right), B=\left(y_2-\lambda c_2\right)$, and $C=\left(y_1-\lambda c_1\right)$. We first notice that if either of $A$ or $C$ is zero, then by using (\ref{linear_system_FB}) iteratively, one obtains $\beta_j=0 \: \forall \: j=1,2,...,N-1$. Thus solutions of $A=0$ or $C=0$ for $\lambda$ do not give eigenvalues. Hence, we consider the cases where $AC\neq0$. We complete the argument by considering two complementary cases.\\

\noindent {\bf Case I:} Assume that $B^{2}-4 A C \neq 0$. Since $\lambda$ is such that $C\neq 0$, the solution-space of \eqref{linear_system_FB} is known to have  dimension $2$ (Theorem $2.21$, \cite{SElaydi05}).  Since $C\neq 0$ and $B^2-4AC \neq 0$, the pair of roots $\nu_{1}$ and $\nu_{2}$ of $A q^{2}+B q+C$ are distinct and nonzero. Using these roots one can directly verify that the general solution of difference scheme $(\ref{linear_system_FB})$ is
\begin{equation}
\label{eq:beta}
\beta_{j}=K_{1} {\nu_{1}}^{j}+K_{2} {\nu_{2}}^{j}, \quad j=0,1, \ldots, N.
\end{equation}
In order to satisfy $\beta_{0}=0$, we must have $K_{1}=-K_{2}\neq 0$. Note that only the nonzero coefficients are relevant as $\beta$ is an eigenvector. Again, to satisfy $\beta_{N}=0$, i.e. $K_{1}\left(\nu_{1}^{N}-\nu_{2}^{N}\right)=0$ we must have $\left(\frac{\nu_{1}}{\nu_{2}}\right)^{N}=1$. 
Therefor, every eigenvalue $\lambda$ of $W$ for which $B^{2}-4 A C \neq 0$, is such that the roots $\nu_{1}$ and $\nu_{2}$ of $A q^{2}+B q+C$ satisfy
$$
\frac{\nu_{1}}{\nu_{2}}=e^{i \frac{2 k \pi}{N}}=\cos \left(\frac{2 k \pi}{N}\right)+i \sin \left(\frac{2 k \pi}{N}\right), \quad \textrm{ for some } k=1,2, \ldots, N-1.
$$
Again since $\nu_{1}$ and $\nu_{2}$ are the solutions of quadratic equation, we also get 
$$
\frac{\nu_{1}}{\nu_{2}}= \frac{-B-\sqrt{B^2-4AC}}{-B+\sqrt{B^2-4AC}} =
\frac{2 B^{2}-4 A C+2 B \sqrt{B^{2}-4 A C}}{4 A C}.
$$
From above two expressions of $\frac{\nu_{1}}{\nu_{2}}$, and (\ref{eq:lemma_1}), we have a simpler relation
$$
\frac{B^{2}}{4 A C}=\cos ^{2}\left(\frac{k \pi}{N}\right), \quad \textrm{ for some } k=1,2, \ldots, N-1.
$$
Thus the eigenvalues can be obtained by solving the following family of equations
\begin{equation}\label{eq:FB_cos_eq_B}
B^{2}-4AC \varphi=0, \quad \varphi=\cos ^{2}\left(\frac{k \pi}{N}\right)  \in[0,1), \quad \textrm{ for some } k=1,2, \ldots, N-1.
\end{equation}
If $\varphi$ happens to be zero, to satisfy \eqref{eq:FB_cos_eq_B}, $B$ must be zero. That corresponds to an eigenvalue $\lambda=\frac{6 c}{5}\left(2 d +\frac{\delta v}{6} \right)$, which is real and positive. For $\varphi\in (0,1)$, substituting the expressions of $A$, $B$, and $C$ in $(\ref{eq:FB_cos_eq_B})$, we obtain after rearrangement
\begin{equation}\label{eq:lambada_B}
(c_2^2 -4 c_1 c_3 \varphi) \lambda^{2} -(2 c_2y_2-4 c_1 y_3\varphi-4c_3y_1 \varphi) \lambda +(y_2^2-4y_1y_3 \varphi)=0.
\end{equation}
Using (\ref{eq:c})-(\ref{eq:y}), the coefficient of $\lambda^2$ is $(c_2^2 -4 c_1 c_3 \varphi)=(25-\varphi+{\delta z}^2 \varphi/4)/(36\delta v^2)$ which is positive as $\varphi\in (0,1)$. Similarly the constant term $(y_2^2-4y_1y_3 \varphi)=c^2\left[4 (1-\varphi)(\frac{1}{{\delta z}^2}+\frac{1}{12})^2+\frac{\varphi}{{\delta z}^2}\right]$ is also positive. On the other hand the coefficient of $\lambda$, that is $-(2  c_2y_2-4 c_1 y_3\varphi-4c_3y_1 \varphi)=- \frac{2 c}{3\delta v}\left(\frac{1}{{\delta z}^2}(5+\varphi) +\frac{5}{12} -\frac{\varphi}{6}\right)$, is a negative quantity. Therefore, both the sum and product of the roots of \eqref{eq:lambada_B} are positive. Hence, the real parts of the roots are positive.

\noindent {\bf Case II:} Assume $B^{2}-4 A C=0$, which gives
\begin{equation}
\label{eq:lambada_FB_B}
(c_2^2 -4 c_1 c_3) \lambda^{2} -(2c_2 y_2 -4 c_1 y_3-4c_3y_1 ) \lambda +(y_2^2-4y_1y_3)=0.
\end{equation}
Again by simplifying the coefficients using (\ref{eq:c})-(\ref{eq:y}), as in Case I, we observe that both the sum and product of the roots of \eqref{eq:lambada_FB_B} are positive. Hence, the real parts of the roots are positive.
\end{proof}

\begin{theorem}\label{lemma:stability_B}
The backward Euler compact scheme (\ref{eq:w_B}) is stable.
\end{theorem}
\begin{proof} 
The proof follows using the assertion of Theorem \ref{theorem:real part}. The argument is analogous to that in the proof of Theorem \ref{theorem:stability_B_v}. We omit the details.
\end{proof}

\subsection{Crank Nicolson compact scheme}
\par Let us define $\tilde{Y}_1=\frac{1}{2}Y_1$. Hence its sub-diagonal, diagonal, and super-diagonal entries are $\tilde{y}_1=\frac{1}{2}y_1$, $\tilde{y}_2=\frac{1}{2}y_2$, and $\tilde{y}_3=\frac{1}{2}y_3$ respectively. 
Taking $\theta=1/2$ in (\ref{eq:oned_fully_discrete_compact}), we get the following fully discrete Crank-Nicolson compact scheme:
\begin{equation}\label{eq:xy}
(X_1+\tilde{Y}_1)U^{m+1}=(X_1-\tilde{Y}_1)U^m+\tilde{F}_1^m,
\end{equation}
where 
$\tilde{F}_1^m=[(c_1-\tilde{y}_{1})h_1(v_m)-(c_1+\tilde{y}_{1})h_1(v_{m+1}),0,...,0,(c_3-\tilde{y}_{3})h_2(v_m)-(c_3+\tilde{y}_{3})h_2(v_{m+1})]$.
Take $\tilde{W}_1=X_1^{-1}\tilde{Y}_1$, then (\ref{eq:xy}) can be written as:
 \begin{equation}\label{eq:w}
(I+\tilde{W}_1)U^{m+1}=(I-\tilde{W}_1)U^{m}+X_1^{-1}\tilde{F}_1^m.
\end{equation}

\begin{theorem}\label{lemma:stability}
The Crank-Nicolson compact scheme (\ref{eq:w}) for convection diffusion equation (\ref{eq:oned_main_equation_2}) is stable.
\end{theorem}
\begin{proof}
Since $2\tilde{W}_1=W_1$, if $\tilde{\lambda}$ is an eigenvalue of $\tilde{W}_1$, then $2\tilde{\lambda}$ is an eigenvalue of $W_1$. Now using Theorem \ref{theorem:real part}, we get that real part of $2\tilde{\lambda}$ is positive. Since, $\tilde{\lambda}$ has positive real part, by following the line of arguments in the proof of Theorem \ref{theorem:stability_v}, we can show that the spectral radius of the amplification matrix for (\ref{eq:w}) is less than 1. Thus the Crank-Nicolson compact scheme (\ref{eq:w}) is stable.
\end{proof}

\section{Condition Number}\label{sec:condition}
In this section, we obtain upper bound for the condition number of the matrix $(I+W).$ First, we bound $\norm{W}_{2},$ where $\norm{.}_{2}$ is spectral norm of a matrix.

\begin{theorem}\label{lemma:wless1_B}
If $\delta z <2$, then we have
$$
\norm{W}_{2}=\mathcal{O}\left(\frac{\delta v}{{\delta z}^2}\right).
$$
\end{theorem}
\begin{proof}
It is enough to show that $\norm{X^{-1}}_2\norm{Y}_2 = \mathcal{O}\left(\frac{\delta v}{{\delta z}^2}\right)$. The upper bound for $\norm{X^{-1}}_2$ will be obtained using the notable Gerschgorin's circle theorem (GCT) [see for example, pp. 61, \cite{Smith78}]. Let $X^*$ be the transpose of $X$, which implies$ (X^{-1})^*X^{-1}=(XX^*)^{-1}$. Hence, following the definition of singular value, $h^2$ is an eigenvalue of $(XX^*)^{-1}$ if and only if $h$ is a singular value of $X^{-1}$. From (\ref{eq_matrix_x_v}), let us write the matrix $P=XX^*$ as follows:
\begin{equation}\label{eq:p_B}
P_{l,l'}= \left\{\begin{array}{ll}
p_l^2+q_l^2+r_l^2 &\textrm{ if }1 < l'=l < N-1\\
q_1^2+r_1^2 &\textrm{ if }l'=l=1 \\
q_{N-1}^2+r_{N-1}^2 &\textrm{ if }l'=l =N-1\\
q_{l\wedge l'}p_{(l\wedge l')+1}+r_{l\wedge l'}q_{(l\wedge l')+1} &\textrm{ if } |l'-l| =1\\
r_{l\wedge l'}p_{(l\wedge l')+2} &\textrm{ if } |l'-l| =2\\
0 &\textrm{ else},
\end{array}\right.
\end{equation}
for all $1 \leq l,l' \leq N-1$. Then $h^{-2}$ is an eigenvalue of the positive definite matrix $P$. Thus
\begin{equation}\label{eq:xinvnorm}
\norm{X^{-1}}_2 =\max\{h\mid h \textrm{ is a singular value of } X^{-1} \}= \frac{1}{\sqrt{\sigma_{min}(P)}},
\end{equation}
where $\sigma_{min}(P)$ is the smallest eigenvalue of $P$. Note that each eigenvalue of $P$ is positive. We obtain the lower bound of eigenvalues of $P$ using GCT. To this end, we consider the following $N-1$ different Gerschgorin's discs to estimate the eigenvalues of matrix $P$:
$$R_1\left(q_1^2+r_1^2,\quad q_1p_2+r_1(p_3+q_2)\right), \quad R_2\left(p_2^2+q_2^2+r_2^2, \quad p_2q_1+q_2(p_3+r_1)+r_2(p_4+q_3)\right),$$ $$R_{N-2}\left(p_{N-2}^2+q_{N-2}^2+r_{N-2}^2, \quad p_{N-2}(q_{N-3}+r_{N-4})+q_{N-2}(p_{N-1}+r_{N-3})+ r_{N-2}q_{N-1}\right),$$ 
$$R_{N-1}\left(p_{N-1}^2+q_{N-1}^2, \quad p_{N-1}(q_{N-2}+r_{N-3} )+q_{N-1}r_{N-2})\right), \textrm{ and-}$$ $$R_l\left(p_l^2+q_l^2+r_l^2, \quad p_l(q_{l-1}+r_{l-2})+q_l(p_{l+1}+r_{l-1})+r_l(p_{l+2}+q_{l+1})\right), \:\:\textrm{for all}\:\: 3 \leq l \leq N-3.$$ 
Here, $R_l(g_l,s_l)$ $1 \leq l \leq N-1$ denotes the $l^{th}$ disc with center $g_l$, and radius $s_l$. If we set $r_0=p_1=r_{N-1}=p_{N}=0$, then for each $1\leq l\leq N-1$, 
$$g_l= p_l^2+q_l^2+r_l^2, \quad\textrm{and } \quad s_l= p_l(q_{l-1}+r_{l-2})+q_l(p_{l+1}+r_{l-1})+r_l(p_{l+2}+q_{l+1}).$$
Clearly from (\ref{eq:c_v}), for sufficiently small $\delta z$, $(g_l-s_l)$ is positive, and of order $\mathcal{O}({\delta v}^{-2})$. Then an application of GCT gives
\begin{align*}
\sigma_{min}(P) \geq \min_{1 \le l\le N-1}(g_l-s_l). 
\end{align*}
Therefore,
\begin{equation}
\label{eq:normx}
\norm{X^{-1}}_2 <  \displaystyle \left(\min_{1 \le l\le N-1} (g_l-s_l) \right)^{-1/2} = \mathcal{O}({\delta v}).
\end{equation}
Let $n_0=l_1=n_{N-1}=l_N=0$, then we have
\begin{equation}
\label{eq:normYin}
\begin{aligned}
    \norm{Y}_{\infty}&=\max_{1 \leq i\leq N-1}\sum_{j=1}^{N-1}\abs{Y_{i,j}}=\max_{1 \leq i\leq N-1}\left(\abs{l_i}+\abs{m_i}+\abs{n_i}\right)\\
    & =\mathcal{O}\left(\frac{1}{{\delta z}^2}\right),
\end{aligned}
\end{equation}
and similarly,
\begin{equation}
\label{eq:normY1}
\begin{aligned}
\norm{Y}_1&=\max_{1 \leq j\leq N-1}\sum_{i=1}^{N-1}\abs{Y_{i,j}} = \max_{1 \leq j\leq N-1}\left(\abs{n_{j-1}}+\abs{m_j}+\abs{l_{j+1}}\right)\\
&=\mathcal{O}\left(\frac{1}{{\delta z}^2}\right).
\end{aligned}
\end{equation}
From the properties of subordinate matrix norm,
\begin{equation}
\label{eq:normy}
\norm{Y}_2 \leq \sqrt{\norm{Y}_1\norm{Y}_{\infty}}.
\end{equation}
Then from above we obtain
$$\norm{Y}_2 = \mathcal{O}\left(\frac{1}{{\delta z}^2}\right).$$
\noindent Therefore, we have
\begin{align*}
\norm{W}_{2}=\norm{X^{-1} Y}_{2} \le
\norm{X^{-1}}_2\norm{Y}_2 &= \mathcal{O}\left(\frac{\delta v}{{\delta z}^2}\right).
\end{align*}
Hence the proof is complete.
\end{proof}
\begin{theorem}\label{theo:condition_num}
If the real part of each eigenvalue of $W$ is positive, the condition number of the matrix $(I+W)$ is $\mathcal{O}\left(\frac{\delta v}{{\delta z}^2}\right)$.
\end{theorem}
\begin{proof}
Note that if $\lambda$ is an eigenvalue of $W$, then $\frac{1}{1+\lambda}$ is an eigenvalue of $(I+W)^{-1}.$ Given that real part of $\lambda$ is positive, we have $\left|\frac{1}{1+\lambda}\right|<1.$ It gives $\norm{(I+W)^{-1}}_{2}<1.$ Thus condition number of matrix $(I+W)=\norm{(I+W)}_{2}\norm{(I+W)^{-1}}_{2}$ is upper bounded by $1+\norm{W}_{2}=\mathcal{O}\left(\frac{\delta v}{{\delta z}^2}\right)$, using Theorem \ref{lemma:wless1_B}. Thus the condition number of $(I+W)$ is $ \mathcal{O}\left(\frac{\delta v}{{\delta z}^2}\right).$
\end{proof}

\begin{remark}
Similarly, the bound on the condition number of matrix $(I+\tilde{W})$ can be derived. Moreover, for the constant coefficient case, the bound on condition number for the matrix $(I+\tilde{W}_1)$ has already been derived in \cite{GoswamiKS}.
\end{remark}

\begin{remark}
A similar computation can be performed for the Crank-Nicolson compact scheme (\ref{eq:w_CN_v}) discussed in Sec. \ref{ssec:cranknicolson} in an analogous way, we have omitted the details for the same. Moreover, for constant coefficient case discussed in Sec. \ref{sec:stability}, the unconditional stability of backward Euler and Crank-Nicolson compact schemes are theoretically established in Theorems \ref{lemma:stability_B}, and \ref{lemma:stability} respectively. In view of that, the numerical verification for the constant coefficient case is omitted. Note that, similar numerical verification for constant coefficient case has been reported extensively by various authors.  
\end{remark}
\section{Numerical Illustrations}\label{sec:numerical}
\par In this section, the stability of backward Euler compact scheme (\ref{eq:w_B_v}), discussed in Sec. \ref{ssec:backward_euler}, is numerically verified for the special choice of the coefficients $a(z)$, and $b(z)$ in Eq. \eqref{eq:oned_main_equation_2_v}. For this specific choice of coefficients and a specific initial-boundary data, we have examined the sharpness of upper bounds, proposed in Section \ref{sec:condition}, on various matrix norms and the condition number of the discretized system of equations (\ref{eq:oned_fully_discrete_compact_v}) numerically. The computation of condition number and its bound for $(I+W)$ are presented for several different choices of $\delta v$ and $\delta z$ so that $\delta v / {\delta z}^2 = 25/32$.

\par To verify theoretical findings, we consider the parameters $a(z)=z+1$, $b(z)=(z+1)^2$, $z_l=0$, $z_r=1$, and $\delta v=0.1$. in Eq. (\ref{eq:oned_main_equation_2_v}). However, other choices of $a(z)$ and $b(z)$ can also be taken. For $N=2,3,\dots,8,$ the expressions for $D_N^1$ from (\ref{eq:d1d2}) are expressed in terms of $A_i,$ $B_i,$ and $C_i,$ $i=1,2,\dots,N-1$ in Table \ref{tab:eigen_test}. Note that for any $N\geq2,$ $D_N^1$ is a polynomial of degree $N-1.$  
Using (\ref{eq:c_v}), (\ref{eq:y_v}), and (\ref{eq:ajbjcj}), the roots of $D_N^1$, i.e., the eigenvalues of matrix $W,$ are computed and listed in Table \ref{tab:eigen_test}. It is observed that all the roots of $D_N^1$ are positive and the smallest root is converging to a limit above 2. Consequently, the desired condition of Theorem \ref{theorem:stability_B_v} is satisfied. Hence, the backward Euler compact scheme (\ref{eq:w_B_v}) for the variable coefficient equation (\ref{eq:oned_main_equation_2_v}) is stable for these parameters.

\par Table \ref{tab:normX} provides an upper bound on $\|X^{-1}\|_2$ for parameters $a(z)=z+1$, $b(z)=(z+1)^2$, $T=1,$ $z_l=0$, and $z_r=1$. For the matrix $X$ given in (\ref{eq_matrix_x_v}), the upper bound of $\|X^{-1}\|_2$ is computed using (\ref{eq:normx}) and listed in third column of Table \ref{tab:normX} for various values of $M$ and $N$. The entries in the fourth column of Table \ref{tab:normX} are obtained from (\ref{eq:xinvnorm}). The upper bound appears reasonably sharp for the given set of parameters. This bound helps us to find an expression of the upper bound on the condition number of $(I+W)$ in terms of discretization parameters.
\par The upper bound on $\|Y\|_2$ for parameters $a(z)=z+1$, $b(z)=(z+1)^2$, $z_l=0$, and $z_r=1$ is provided in Table \ref{tab:normY}. For the matrix $Y$ given in (\ref{eq_matrix_y_v}), the entries of the second and third columns in Table \ref{tab:normY} are computed using (\ref{eq:normYin}) and (\ref{eq:normY1}), respectively for various values of $N$. The upper bound of $\|Y\|_2$ is obtained using (\ref{eq:normy}) and listed in the fourth column of Table \ref{tab:normY}. The entries in fifth column are obtained from the expression $\|Y\|_2=\sqrt{\lambda_{max}(Y^*Y)}$, where $\lambda_{max}$ is the maximum eigenvalue of $Y^{*}Y$. The upper bound on $\|Y\|_2$ also plays an important role in studying the condition number of $(I+W).$
\par Note that the condition number of $(I+W)$ is crucial in solving (\ref{eq:w_B_v}). Utilizing the upper bounds on $\|X^{-1}\|_2$ and $\|Y\|_2,$ an upper bound on the condition number of $(I+W)$ is obtained by following Theorem \ref{theo:condition_num} for parameters $a(z)=z+1$, $b(z)=(z+1)^2$, $T=1,$ $z_l=0$, and $z_r=1$. Table \ref{tab:contion_num} provides the numerical value of the upper bound of the condition number of $(I+W)$ for some choices of $N$ and $M$ so that $\frac{\delta v}{{\delta z}^2}$ is constant. It is observed that the upper bound on the condition number is reasonably small, which asserts the robustness of the proposed numerical scheme. Moreover, in view of the last column of Table \ref{tab:contion_num}, the derived upper bound is reasonably sharp also.

\begin{table}[h!]
    \centering
    \begin{tabular}{c|c|c}
        \hline
        & $D^1_N$& Roots of $D^1_N$\\
        \hline \hline
        &&\\
        $N=2$& $-B_1$& 2.0600\\
        \hline 
        &&\\
        $N=3$& $B_1B_2 - A_1C_2$ & 7.9419\\
        &  & 2.1218\\
        \hline
        &&\\
        $N=4$& $A_2B_1C_3+A_1B_3C_2-B_1B_2B_3$  & 17.7303\\
        &  &8.0194\\ 
        &  &2.1423\\
        \hline
        &&\\
        $N=5$& $A_1A_3C_2C_4 - A_1B_3B_4C_2 - A_2B_1B_4C_3$ &31.4791\\
        & $ - A_3B_1B_2C_4 + B_1B_2B_3B_4$ &17.5329\\
        & &8.1618\\
        & &2.1491\\
        \hline
        &&\\
        $N=6$& $A_1B_3B_4B_5C_2 - A_1A_4B_3C_2C_5 - A_1A_3B_5C_2C_4$ & 49.3158\\
        & $- A_2A_4B_1C_3C_5 + A_2B_1B_4B_5C_3 + A_3B_1B_2B_5C_4$ &  30.5320\\
        & $+A_4B_1B_2B_3C_5 - B_1B_2B_3B_4B_5$ &17.8975\\
        & & 8.2353\\
        & &2.1517\\ 
        \hline
        &&\\
        $N=7$& $A_1A_3B_5B_6C_2C_4 - A_1A_3A_5C_2C_4C_6+A_1A_4B_3B_6C_2C_5$ & 71.3459\\
        & $ + A_1A_5B_3B_4C_2C_6-A_1B_3B_4B_5B_6C_2 + A_2A_4B_1B_6C_3C_5$ &47.0817\\
        & $-A_1B_3B_4B_5B_6C_2 + A_2A_4B_1B_6C_3C_5 + A_2A_5B_1B_4C_3C_6$ &31.0721\\
        & $- A_2B_1B_4B_5B_6C_3 + A_3A_5B_1B_2C_4C_6 - A_3B_1B_2B_5B_6C_4$ &18.1722\\
        & $- A_4B_1B_2B_3B_6C_5 - A_5B_1B_2B_3B_4C_6 + B_1B_2B_3B_4B_5B_6$ &8.2712\\
        & &2.1528\\
        \hline
        &&\\
        $N=8$& $A_1A_3A_5B_7C_2C_4C_6 + A_1A_3A_6B_5C_2C_4C_7 - A_1A_3B_5B_6B_7C_2C_4$&97.6417\\
        & $+ A_1A_4A_6B_3C_2C_5C_7 - A_1A_4B_3B_6B_7C_2C_5 - A_1A_5B_3B_4B_7C_2C_6$& 67.3382\\
        &$- A_1A_6B_3B_4B_5C_2C_7 + A_1B_3B_4B_5B_6B_7C_2 + A_2A_4A_6B_1C_3C_5C_7$& 47.5424\\
        &$- A_2A_4B_1B_6B_7C_3C_5 - A_2A_5B_1B_4B_7C_3C_6 - A_2A_6B_1B_4B_5C_3C_7$&31.6907\\
        &$+ A_2B_1B_4B_5B_6B_7C_3 - A_3A_5B_1B_2B_7C_4C_6 - A_3A_6B_1B_2B_5C_4C_7$& 18.3334\\
        &$+ A_3B_1B_2B_5B_6B_7C_4 - A_4A_6B_1B_2B_3C_5C_7 + A_4B_1B_2B_3B_6B_7C_5$& 8.2896\\
        &$+ A_5B_1B_2B_3B_4B_7C_6 + A_6B_1B_2B_3B_4B_5C_7 - B_1B_2B_3B_4B_5B_6B_7$ & 2.1534\\
        \hline
    \end{tabular}
    \caption{Expressions and roots of the polynomial $D^1_N$ for different choices of $N,$ with parameters $a(z)=z+1$, $b(z)=(z+1)^2$, $z_l=0$, $z_r=1$, and $\delta v=0.1$.}
    \label{tab:eigen_test}
\end{table}
\begin{table}
    \centering
    \begin{tabular}{c|c|c|c}
    \hline
        $N$ & $M$ & Upper bound on &   $\|X^{-1}\|_{2}$ \\
        & & $\|X^{-1}\|_{2},$ using (\ref{eq:normx})&\\
    \hline \hline 
        25 & 800 &$1935.87\times10^{-6}$ & $1870.88\times 10^{-6}$ \\
        50 & 3200 & $4840.86\times 10^{-7}$ & $4684.92\times 10^{-7}$\\
        100 & 12800 & $1210.29\times 10^{-7}$  & $1171.71\times 10^{-7}$ \\
        200 & 51200 & $3025.75\times 10^{-8}$ & $2929.59\times10^{-8}$\\
        400 & 204800 & $7564.41\times 10^{-9}$ & $7324.16\times 10^{-9}$ \\
        800 & 819200 & $1891.10\times10^{-9}$ & $1831.05\times 10^{-9}$ \\
    \hline
    \end{tabular}
    \caption{Upper bound on $2-$norm of $X^{-1}$ with parameters $a(z)=z+1$, $b(z)=(z+1)^2$, $T=1,$ $z_l=0$, and $z_r=1$. The entries in fourth column are obtained from (\ref{eq:xinvnorm}).}
    \label{tab:normX}
\end{table}

\begin{table}
    \centering
    \begin{tabular}{c|c|c|c|c}
    \hline
        $N$ & $\|Y\|_{\infty}$ &   $\|Y\|_{1}$ & Upper bound on  &  $\|Y\|_2$\\
        & & & $\|Y\|_2$, using (\ref{eq:normy})& \\
    \hline \hline 
        25 & 9214.83 & 9217.33 & 9215.83 & 8373.84\\
        50 & 38414.33 & 38417.33 & 38415.83 & 35813.62\\
        100 & 156814.33  & 156817.33 & 156815.83 &149308.18 \\
        200 & 633614.33 & 633617.33 & 633615.83 & 612829.53\\
        400 & 2547214.33 & 2547217.33 & 2547215.83 & 2491169.46\\
        800 & 10214414.33 & 10214417.33  & 10214415.83 & 10065975.94\\
    \hline
    \end{tabular}
    \caption{Upper bound on $2-$norm of $Y$ with parameters $a(z)=z+1$, $b(z)=(z+1)^2$, $z_l=0$, and $z_r=1$. The entries in fifth column are $\|Y\|_2=\sqrt{\lambda_{max}(Y^*Y)}$, where $\lambda_{max}$ is the maximum eigenvalue of $Y^{*}Y$.}
    \label{tab:normY}
\end{table}
\begin{table}
    \centering
    \begin{tabular}{c|c|c|c|c|c}
    \hline
         $N$ & $M$ & Upper bound & Upper bound & Upper bound for & $\kappa(I+W)$ \\
        & & of $\|X^{-1}\|_{2}$ & of $\|Y\|_2$  & $\kappa(I+W)$&  \\
        & & & & from Theorem \ref{theo:condition_num}&\\
    \hline \hline 
        25 & 800 & $1935.87\times10^{-6}$  & 9215.83 & 18.84 & 15.93\\
        50 & 3200 & $4840.86\times 10^{-7}$ & 38415.83 & 19.60 & 17.42\\
        100 & 12800 & $1210.29\times 10^{-7}$  & 156815.83 & 19.98 & 18.30\\
        200 & 51200 & $3025.75\times 10^{-8}$ & 633615.83 & 20.17 &18.84\\
        400 & 204800 & $7564.41\times 10^{-9}$ & 2547215.83 & 20.27 & 19.17\\
        800 & 819200 & $1891.10\times10^{-9}$ & 10214415.83 & 20.32 &  19.39\\
    \hline
    \end{tabular}
    \caption{Upper bound for $\kappa(I+W)$, i.e. condition number of $(I+W)$, with parameters $a(z)=z+1$, $b(z)=(z+1)^2$, $T=1,$ $z_l=0$, and $z_r=1$. The condition numbers given in the last column are computed using MATLAB.}
    \label{tab:contion_num}
\end{table}

\section{Conclusions and future directions}\label{sec:conclu}
A novel difference equation based approach has been developed to investigate the stability of variable coefficient convection diffusion equations using compact schemes. It has been observed that the compact schemes for variable coefficient problems are stable under certain restrictions. Further, the unconditional stability of compact schemes for the constant coefficient case has been proved theoretically. An estimate on the condition number of amplification matrix has also been derived. The MATLAB algorithms have been developed, and a supporting numerical example has been provided. As a future scope, the proposed approach can be extended to the multi-dimensional problems, system of PDEs, etc. This work finds an application to investigate the stability of the variable coefficient PDEs which can not be transformed to constant coefficient problems.

\section*{Acknowledgments}
Authors acknowledge the financial support from the National Board for Higher Mathematics (NBHM), Department of Atomic Energy, Government of India, under the grant number 02011-32-2023-R$\&$D-II-13347 and the Department of Science and Technology under the Project-Related Personal Exchange Grant DST/INTDAAD/P-12/2020. 

\appendix
\section{MATLAB codes}\label{sec:appendixA}

\subsection{Expression of $D^1_N$ } \label{ssec:1} The following algorithm defines a function that gives symbolic expression of $D^1_N$.

\begin{lstlisting}[frame=single]
function[D1N]=betaandD(N)
syms(sym('A',[1 N-1]));
A=(sym('A',[1 N-1]));
syms(sym('B',[1 N-1]));
B=(sym('B',[1 N-1]));
syms(sym('C',[1 N-1]));
C=(sym('C',[1 N-1]));
syms(sym('P',[1 N+1]));
P=(sym('P',[1 N+1]));
P(1)=1;P(3)=-B(1)*P(2)-C(1);
j=4;
for i=1:N-2
P(j)=(-C(j-2))*A(j-3)*P(j-2)+(-B(j-2))*P(j-1);
j=j+1;
end
K=coeffs(P(N+1),P2);
D1N=simplifyFraction(K(2)); 
end
\end{lstlisting}

\subsection{Roots of $D^1_N$}\label{ssec:2}  The following algorithm defines a function that gives the symbolic computation of the roots of $D^1_N$.
\begin{lstlisting}[frame=single] 
function[u]=roots_value(N,p,q,r,l,m,n)
syms lambda
syms(sym('A',[1 N-1]));
A=(sym('A',[1 N-1]));
syms(sym('B',[1 N-1]));
B=(sym('B',[1 N-1]));
syms(sym('C',[1 N-1]));
C=(sym('C',[1 N-1]));
for h=1:N-1
A(h)=n(h)-lambda*r(h);B(h)=m(h)-lambda*q(h);C(h)=l(h)-lambda*p(h);
end
syms(sym('P',[1 N+1]));
P=(sym('P',[1 N+1]));
P(1)=1;P(3)=-B(1)*P(2)-C(1);j=4;
for i=1:N-2
P(j)=(-C(j-2))*A(j-3)*P(j-2)+(-B(j-2))*P(j-1);
j=j+1;
end
K=coeffs(P(N+1),P2); 
D1N=simplifyFraction(K(2)); 
M=coeffs(D1N,lambda);V=zeros(N,1);
for b=1:N
V(b)=M(N-b+1);
end
u=roots(V); 
end
\end{lstlisting}

\subsection{Positivity test of roots of $D^1_N$} \label{ssec:3}  The following algorithm reads the coefficients and parameters associated with the initial-boundary value problem \eqref{eq:oned_main_equation_2_v}, \eqref{eq:initial_condition_1d_v}, \eqref{eq:oned_boun_v} and the related compact difference scheme. This computes the numerical values of the roots of $D^1_N$ and conveys if all are positive or not.
\begin{lstlisting}[frame=single]
N=input('Enter value of N=');
z_l=input('Enter value of left space boundary point z_l=');
z_r=input('Enter value of right space boundary point z_r=');
dv=input('Enter value of time step dv=');
syms a(z) b(z) w1(z) y1(z) w2(z) y2(z)
a(z)=input('Enter the convection coefficient a(z)=');
b(z)=input('Enter the diffusion coefficient b(z)=');
x=linspace(z_l,z_r,N+1);dz=x(2)-x(1); 
w=a(x);y=b(x);w1(z)=diff(a(z));y1(z)=diff(b(z));
w2(z)=diff(w1(z));y2(z)=diff(y1(z)); 
wd=w1(x);yd=y1(x);wdd=w2(x);ydd=y2(x);
gamma=zeros(N-1,1);xeta=zeros(N-1,1);alpha=zeros(N-1,1);
p=zeros(N-1,1);q=zeros(N-1,1);r=zeros(N-1,1); 
l=zeros(N-1,1);m=zeros(N-1,1);n=zeros(N-1,1);
for i=1:N-1 
gamma(i)=(w(i+1)/y(i+1))+(2/y(i+1))*yd(i+1);
xeta(i)=w(i+1)-((dz^2)/12)*((w(i+1)/y(i+1))*wd(i+1)+(2/y(i+1))*wd(i+1)*yd(i+1)-wdd(i+1));
alpha(i)=y(i+1)+((dz^2)/12)*(((w(i+1)/y(i+1))*yd(i+1))-2*wd(i+1)+ydd(i+1)-((2/y(i+1))*(yd(i+1)^2))+((w(i+1)^2)/y(i+1)));
p(i)=((2+dz)*gamma(i))/(24*dv);
q(i)=5/(6*dv);
r(i)=((2-dz)*gamma(i))/(24*dv);
l(i)=(-xeta(i)/(2*dz))-(alpha(i)/(dz^2));
m(i)=(2*alpha(i))/(dz^2);
n(i)=(xeta(i)/(2*dz))-(alpha(i)/(dz^2));
end
[D1N]=betaandD(N);
fprintf('The expression for D1_N is= %s .\n',D1N);
[u]=roots_value(N,p,q,r,l,m,n);
if min(real(u)) > 0
disp('Result: Since the roots are positive, the compact scheme is stable for above parameter')
else
disp('Result: Since the roots are NOT positive, Stability condition is not met for above parameter')
end
\end{lstlisting}

\bibliographystyle{elsarticle-num}
\bibliography{references}
\end{document}